\documentclass{amsart}
\usepackage{amsfonts,amssymb,amsmath,amsthm}
\usepackage{url}
\usepackage{enumerate}

\newtheorem{thm}{Theorem}[section]

\newtheorem{defn}{Definition}[section]
\newtheorem{prop}{Proposition}[section]

\newtheorem{lem}{Lemma}[section]

\newtheorem{exmpl}{Example}[subsection]

\numberwithin{equation}{section}

\author{RB Yadav}
\address{
Department of Mathematics\\
  IIT Guwahati\\
 Assam 781039, INDIA}
\email{rbyadav15@gmail.com}
\author{Srikanth KV}
\address{
 Department of Mathematics\\
  IIT Guwahati\\
 Assam 781039, INDIA}
\email{kvsrikanth@iitg.ernet.in}

\begin{document}

\title[ON STUDY OF A METRIC  ON $C(S^1,S^1)$ ]{ON STUDY OF A METRIC  ON $C(S^1,S^1)$ }

\begin{abstract}
In this article we  define a metric  on $C(S^1,S^1)$. Also, we give some density results in $C(S^1,S^1)$.
\end{abstract}
\maketitle
\section{Introduction}
M. H. Stone gave the Stone-Weierstrass Theorem \cite{3}, which is a density result
with respect to uniform topology. This result is a generalization of the Weierstrass
Approximation Theorem by lightening the restrictions imposed on the domain over
which the given functions are defined. By taking the co-domain of the given functions
to be the complex plane in place of real line, he went further and gave a complex
version of the latter result as Stone-Weierstrass Theorem \cite{4}. In this paper our aim is to
seek density results for $C(S^1,S^1)$, the class of continuous functions from the unit circle $S^1$, given by $S^1=\{(x,y)\in \mathbb{R}^2 : x^2+y^2=1\}$, to itself.

\section{Preliminaries}
In this section we recall some basic notions.
Let $X$ be a compact houdorff space.  Define

$$C(X,\mathbb{R}) = \{f:X\to \mathbb{R}: \text{f is continuous}\}.$$
\begin{defn}
\begin{enumerate}
  \item  An algebra is a bimodule together with a bilinear product.
  \item If $\mathcal{A}$ and $\mathcal{A}_1$ are algebras such that $\mathcal{A}_1\subset \mathcal{A}$ and module and product structures on $\mathbb{A}_1$ are the induced  ones  from $\mathcal{A}.$
\end{enumerate}

\end{defn}
\begin{exmpl}
  \begin{enumerate}
    \item $\mathbb{R}$ and $\mathbb{C}$ are algebras.
    \item $C(X,\mathbb{R})$ and $C(X,\mathbb{C})$ are algebras.
    \item Set of polynomials is subalgebra of $C([a,b],\mathbb{R})$.
  \end{enumerate}
\end{exmpl}
  \begin{defn}
    \begin{enumerate}
      \item A subalgebra $\mathcal{A}\subset C(X,\mathbb{R})$ is said to be unital if $1\in \mathcal{A}.$
      \item  $\mathcal{A}\subset C(X,\mathbb{R})$ is said to separate points of $X$ if $\forall a,b\in X$ with $a\ne b$ there exists $f\in \mathcal{A}$ such that $f(a)\ne f(b).$
    \end{enumerate}
  \end{defn}
\begin{thm}[Stone-Weierstrass theorem]\label{thm:swcreal}
 If $\mathcal{A}$ is a unital sub--algebra of $C(X,\mathbb{R})$
which separates points then $\mathcal{A}$ is dense in $C(X,\mathbb{R})$ in the topology induced by $\sup$ metric.
\end{thm}

Next  we recall some basic notions related to  covering spaces.\\
Let $E$ and $X$ be topological spaces, and let $q:E\to X$ be a continuous map.
\begin{defn}
An open set $U\subset X$ is said to be \textbf{evenly covered by q} if $q^{-1} (U)$ is a disjoint union of connected open subsets of $E$ (called the \textbf{ sheets of the covering over $U$}), each of which is mapped homeomorphically onto $U$ by $q$.
\end{defn}
\begin{defn}
A \textbf{covering map or projection map} is a continuous surjective map $q:E\to X$  such that $E$ is connected and locally path-connected, and every point of $X$ has an evenly covered neighborhood. If $q:E\to X$ is a covering map, we call $E$ a \textbf{covering space of $X$}  and $X$ the \textbf{base of the covering}.
\end{defn}
\begin{exmpl}\label{exp}
The exponential quotient map $p:\mathbb{R}\to S^1$ given by $p(x)=\exp (2\pi \imath x)$ is a covering map.
\end{exmpl}
\begin{exmpl}
The $n^\mathrm{th}$ power map $p_n:S^1\to S^1$  given by  $p_n(z)=z^n$ is also a covering map.
\end{exmpl}
\begin{exmpl}
Let $\mathbb{T}^n=\underbrace{S^1\times S^1\times ...\times S^1}_\text{n times}$. Define $p^n:\mathbb{R}^n\to \mathbb{T}^n$ by
$$p^n(x_1,....,x_n)=(p(x_1),...,p(x_n)),$$
where $p$ is exponential map of Example \ref{exp}. It can be verified that $p^n$ is a covering map.
\end{exmpl}
\begin{defn}
If $q:E\to X$ is a covering map and $\phi :Y\to X$ is any continuous map, a \textbf{lift of $\phi$ } is a continuous map $\tilde{\phi}:Y\to E$ such that $q\circ \tilde{\phi}=\phi$.
\end{defn}
From \cite{2} we have following result called as \textbf{Path Lifting Property}:
\begin{lem}\label{raj100}
Let $q:E\to X$ be a covering map. Suppose $f:[0,1]\to X$ is any path, and $e\in E$ is any point in the fiber of $q$ over $f(0)$. Then there exists a unique  lift $\tilde{f}:I\to E$ of $f$ such that $\tilde{f}(0)=e$.
\end{lem}
\section{A metric on $C(S^1,S^1)$}\label{DS^1}
Let $C(S^1,S^1)$ denote the collection of all continuous functions from $S^1$ to $S^1$. Further let $I=[0,1]$ and
\begin{footnotesize}$\widehat{C}(I,\mathbb{R})=\{f|f:I\to \mathbb{R}$ is continuous  with  $f(0)\in [0,1)$\;and\; $f(1)-f(0)\in\mathbb{Z}\}$.
\end{footnotesize}
Consider the function $\alpha :[0,1]\rightarrow S^1$ given by $\alpha(t)=(\cos 2\pi t,\sin 2\pi t)$ and the covering space $\mathbb{R}$ of $S^1$ with projection map $p$ given by $p(t)=(\cos 2\pi t,\sin 2\pi t)$.\\
For $f\in C(S^1,S^1)$, let $\tilde{f_{\alpha}}$ be the unique lift \ref{raj100} of $f\circ\alpha$ such that $\tilde{f_{\alpha}}(0)\in [0,1)$. Clearly $\tilde{f_{\alpha}}\in \widehat{C}(I,\mathbb{R})$.\\
For $f\in C(S^1,S^1)$ and $g \in C(S^1,S^1)$, define
\begin{center}
 $d_0(f,g)= 2\pi \sup_{\substack{x\in S^1\\x\ne(1,0)}}|\tilde{f_{\alpha}}(\alpha^{-1}(x))-\tilde{g_{\alpha}}(\alpha^{-1}(x))|$\\
\end{center}
\noindent Let  $\widehat{C}(I,\mathbb{R})$ be equipped with the metric $d_1$ induced by $\sup$  norm on it.
\begin{prop}$d_0$ is a metric  on $C(S^1,S^1)$ and there exists a homeomorphism $\phi:C(S^1,S^1)\rightarrow \widehat{C}(I,\mathbb{R})$ such that $d_0(f,g)=2\pi d_1(\phi(f),\phi(g))$, for all f and g in $C(S^1,S^1)$ .
 \end{prop}
 \begin{proof}
 Define $\phi:C(S^1,S^1)\rightarrow \bar{C}(I,\mathbb{R})$  by
 $\phi(f)(t)=\tilde{f_{\alpha}}(t), \forall t\in I$\\
 Now, $\forall t\in I$, $\phi(f)=\phi(g)\Rightarrow\tilde{f_{\alpha}}(t)= \tilde{g_{\alpha}}(t) \Rightarrow p\circ\tilde{f_{\alpha}}(t)= p\circ\tilde{g_{\alpha}}(t)\Rightarrow f\circ\alpha(t)= g\circ\alpha(t)$. Again, $\forall t\in I, f\circ\alpha(t)= g\circ\alpha(t)\Rightarrow f(x)= g(x)$,  $\forall x\in S^1 \Rightarrow f= g$.
So $\phi$ is injective.\\
Given any $g\in \widehat{C}(I,\mathbb{R})$, define $f$ by \\
\begin{equation*}
f(x)=\left\{
\begin{array}{rl}
p\circ g\circ \alpha^{-1}(x) & \text{if  $x\in S^1-\{(1,0)\}$}\\
p\circ g(0)=p\circ g(1) & \text{otherwise}
\end{array}\right.
\end{equation*}
Clearly $f\in C(S^1,S^1)$ and $\phi(f)=g$. So $\phi$ is onto.
\begin{align*}
d_0(f,g)&= 2\pi \sup_{\substack{x\in S^1\\x\ne(1,0)}}|\tilde{f_{\alpha}}(\alpha^{-1}(x))-\tilde{g_{\alpha}}(\alpha^{-1}(x))|\\
        &= 2\pi \sup_{\substack{x\in S^1\\x\ne(1,0)}}|\phi(f)(\alpha^{-1}(x))-\phi(g)(\alpha^{-1}(x))|\\
         &= 2\pi \sup_{t\in(0,1)}|\phi(f)(t)-\phi(g)(t)|\\
         &= 2\pi \sup_{t\in[0,1]}|\phi(f)(t)-\phi(g)(t)|\\
         &=2\pi d_1(\phi(f),\phi(g))
\end{align*}
So $d_0$ is a metric space and $\phi$ is a homeomorphism.
\end{proof}
\subsection{Density results  using  PL functions  and  polynomials}
 Recall that a function $f:[a,b]\rightarrow \mathbb{R}$ is said to be  piecewise linear (PL) if there exist sets of points $\{p_i\}_{i=1}^k$ in $[a,b]$ and $\{q_i\}_{i=1}^k$ in $\mathbb{R}$ such that $p_1=a$, $p_k=b$ and $\forall t\in [p_i,p_{i+1}]$, $f(t)= q_i +(q_{i+1}-q_i)\frac{(t-p_i)}{p_{i+1}-p_i}$,  for every  $1\le i\le k-1$.\\
For $a,b\in \mathbb{R}$ define

$${PL}_{a,b}(I,\mathbb{R})=\{f|f:I\rightarrow \mathbb{R} \;\mathrm{ is\; PL \; and}\; f(0)=a,\; f(1)=b\}$$
 $$C_{a,b}(I,\mathbb{R})=\{f\in C(I,\mathbb{R})|f(0)=a,\; f(1)=b\}$$
 $$P_{a,b}(I,\mathbb{R})=\{p(x)|p(x)\; \mathrm{is\; a\; polynomial\; on}\; \mathbb{R}\;\mathrm{ with}\; p(0)=a,\;  p(1)=b\}$$
Note that ${PL}_{a,b}(I,\mathbb{R})$ is dense in $C_{a,b}(I,\mathbb{R})$.
\begin{lem}
Given $a\in \mathbb{R}$, the  set $P_{0,a}(I,\mathbb{R})$ is dense in $C_{0,a}(I,\mathbb{R})$ with $\sup$ metric.
\end{lem}
\begin{proof}
Let $f\in C_{0,a}(I,\mathbb{R})$. By Weierstrass approximation theorem  a sequence of polynomials $\{p_n\}$ such that $p_n\to f$ uniformly. Now define $\tilde{p}_n$ by\\ $\tilde{p}_n(x)= p_n(x)-p_n(0)+ x(f(1)-p_n(1)+p_n(0))$. \\
Clearly $\tilde{p}_n\to f$ uniformly and $\tilde{p}_n(0)=0$, $\tilde{p}_n(1)=a$.
\end{proof}
\begin{thm}\label{marker} Given $a,b\in \mathbb{R}$, the set $P_{a,b}(I,\mathbb{R})$ is dense in $C_{a,b}(I,\mathbb{R})$ with $\sup$ metric.
\end{thm}
\begin{proof}
Define $\psi:C_{a,b}(I,\mathbb{R})\to C_{0,b-a}(I,\mathbb{R})$ via
$\psi(f)(x)=f(x)-a$\\
Clearly $\psi $ is a homeomorphism and $\psi(P_{a,b}(I,\mathbb{R}))=P_{0,b-a}(I,\mathbb{R})$. So, since $P_{0,b-a}(I,\mathbb{R})$ is dense in $C_{0,b-a}(I,\mathbb{R})$, $P_{a,b}(I,\mathbb{R})$ is dense in $C_{a,b}(I,\mathbb{R})$.
\end{proof}
\begin{defn}
Let $f:S^1\to S^1$ be a continuous function. Then winding number $W(f)$ of $f$ is the integer $\tilde{f_{\alpha}}(1)-\tilde{f_{\alpha}}(0)$.
\end{defn}
 For $q\in S^1$ and $m\in\mathbb{Z}$ define
\begin{footnotesize}
$$C_m^q(S^1,S^1)=\{f\in C(S^1,S^1)|f((1,0))=q \;\mathrm{and}\;  W(f)=m \}$$
$$ PL_m^q(S^1,S^1)=\{ f\in C(S^1,S^1)|\tilde{f_{\alpha}}\in PL_{\tilde{q},\tilde{q}+m}(I,\mathbb{R}),\;p\circ\tilde{f_{\alpha}}(0)=q\}$$
$$P_m^q(S^1,S^1)=\{ f\in C(S^1,S^1)|\tilde{f_{\alpha}}\in P_{\tilde{q},\tilde{q}+m}(I,\mathbb{R}),\;p\circ\tilde{f_{\alpha}}(0)=q\}.$$
\end{footnotesize}
\begin{thm}\label{rb}
 The set  $ PL_m^q(S^1,S^1)$ is dense in $C_m^q(S^1,S^1)$.
\end{thm}
\begin{proof}
If $\tilde{f_{\alpha}}(0)=\tilde{q}$, then clearly $\phi^{-1}(C_{\tilde{q},\tilde{q}+m}(I,\mathbb{R}))=C_m^q(S^1,S^1)$ and \\ $\phi^{-1}(PL_{\tilde{q},\tilde{q}+m}(I,\mathbb{R}))=PL_m^q(S^1,S^1)$.\\
So, since $PL_{\tilde{q},\tilde{q}+m}(I,\mathbb{R})$ is dense in $C_{\tilde{q},\tilde{q}+m}(I,\mathbb{R})$ and $\phi $ is a homeomorphism, $PL_m^q(S^1,S^1)$ is dense in $C_m^q(S^1,S^1)$.
\end{proof}
\begin{thm}The set  $ P_m^q(S^1,S^1)$ is dense in $C_m^q(S^1,S^1)$.
\end{thm}
\begin{proof}
Proof is similar to the proof of Theorem\ref{rb}.
\end{proof}
\section{Stone-Weierstrass type theorem for $C(S^1,S^1)$}
Let X be a compact metric  space. Fix $u,v\in X$ with $u\ne v$, $a,b \in \mathbb{R}$ and $\mathcal{A}\subset C(X,\mathbb{R})$. Define
\begin{center}
$C_{u,v}^{a,b}(X,\mathbb{R})=\{f\in C(X,\mathbb{R})|f(u)=a,\; f(v)=b\}$\\
$\mathcal{A}_{u,v}^{a,b}(X,\mathbb{R})=\{f\in\mathcal{A}|f(u)=a,\; f(v)=b\}.$
\end{center}
\begin{lem}
If $\mathcal{A}$ is a unital sub-algebra of $C(X,\mathbb{R})$ which separates points and $a$, $b$ are any two distinct  real numbers then  $\mathcal{A}_{u,v}^{a,b}(X,\mathbb{R})$ is dense in $C_{u,v}^{a,b}(X,\mathbb{R})$.
\end{lem}
\begin{proof}
Since  $\mathcal{A}$ is a unital sub-algebra of $C(X,\mathbb{R})$ which separates points, by Stone-Weierstrass theorem  \cite{3} $\mathcal{A}$ is dense in $C(X,\mathbb{R})$.
So, for each $f\in C(X,\mathbb{R})$ there exists a sequence $f_n$ in  $\mathcal{A}$ such that $f_n\to f$ uniformly. Since $a\ne b$ we can choose $f_n$ in such a way that  $f_n(v)-f_n(u)\ne 0, \forall n$.\\
Now define
\begin{center}
$\tilde{f_n}=a+\frac{(b-a)(f_n(x)-f_n(u))}{f_n(v)-f_n(u)}$
\end{center}
Clearly $\tilde{f_n}(u)=a$ and $\tilde{f_n}(v)=b$, $\forall n$ and $\tilde{f_n}\to f$ uniformly.
\end{proof}
\begin{thm}\label{raj102}
If $\mathcal{A}$ is a unital sub-algebra of $C(X,\mathbb{R})$ which separates points and $a$, $b$ are any two   real number then  $\mathcal{A}_{u,v}^{a,b}(X,\mathbb{R})$ is dense in $C_{u,v}^{a,b}(X,\mathbb{R})$.
\end{thm}
\begin{proof}
When $a\ne b$ the result is true by above lemma. So let $f\in C_{u,v}^{a,a}(X,\mathbb{R})$. First let $a\ne 0$. By Urysohn Lemma there exists  $g \in C(X,\mathbb{R})$ such that $g(u)=-\frac{a}{4}$ and  $g(v)=\frac{a}{4}$.
Now define $f_1$ and $f_2$ by $f_1=\frac{1}{2}f+g$ and $f_2=\frac{1}{2}f-g$. Clearly $f_1\in C_{u,v}^{\frac{1}{4}a,\frac{3}{4}a}(X,\mathbb{R})$ and $f_2\in C_{u,v}^{\frac{3}{4}a,\frac{1}{4}a}(X,\mathbb{R})$ and $f_1+f_2=f$. So by above lemma  there exist sequences $f_{1,n}\in \mathcal{A}_{u,v}^{\frac{1}{4}a,\frac{3}{4}a}(X,\mathbb{R})$ and $f_{2,n}\in \mathcal{A}_{u,v}^{\frac{3}{4}a,\frac{1}{4}a}(X,\mathbb{R})$ such that $f_{1,n}\to f_1$ and $f_{2,n}\to f_2$ uniformly. Now  define $f_n =f_{1,n}+f_{2,n}$. Clearly $f_n\in \mathcal{A}_{u,v}^{a,a}(X,\mathbb{R})$ and $f_n\to f$ uniformly. Again  proof for the case when $a=0$, can be given  in the same way as was given in Theorem \ref{marker}.
\end{proof}
Let $X$ be a compact housdorff space. Let  $\mathbb{F}$ denotes either the field of real numbers $\mathbb{R}$ or the
field of complex numbers $\mathbb{C}$,  $C(X,\mathbb{F})$ denotes the collection of
$\mathbb{F}$--valued continuous functions on $X$ with the sup--norm. And, let $\mathcal{A}$ denote
a subset of $C(X,\mathbb{F})$.

Let $k$ be any natural number. Let $S=\{ x_1,x_2,\cdots,x_k\}\subset X$ (with $x_i\neq x_j$ for distinct $i$ and $j$)
and let $V=\{ v_1, v_2, \cdots, v_k\}\subset \mathbb{F}$. Define
\begin{align*}
C_S^V(X,\mathbb{F}) &= \{f\in C(X,\mathbb{F}) \;| \;f(x_1) =v_1, f(x_2)=v_2, \cdots, f(x_k)=v_k \}\mathrm{\;\;and}\\
\mathcal{A}_S^V(X,\mathbb{F}) &= \{f\in \mathcal{A} \;| \;f(x_1) =v_1, f(x_2)=v_2, \cdots, f(x_k)=v_k \}.
\end{align*}


Theorem \ref{raj102} is a particular case of the following theorem from \cite{1}.
\begin{thm}[Stone-Weierstrass with finitely many interpolatory constraints]\label{thm:swcreal}
For a natural number $k$, let $S=\{ x_1,x_2,\cdots,x_k\}\subset X$ (with $x_i\neq x_j$ for distinct $i$ and $j$)
and let $V=\{ v_1, v_2, \cdots, v_k\}\subset \mathbb{R}$. If $\mathcal{A}$ is a unital sub--algebra of $C(X,\mathbb{R})$
which separates points then $\mathcal{A}_S^V(X,\mathbb{R})$ is dense in $C_S^V(X,\mathbb{R})$.
\end{thm}

For $q\in S^1$,  $m\in\mathbb{Z}$  and a unital sub-algebra $\mathcal{A}$ of $C(I,\mathbb{R})$ which separates points, define
\begin{center}
\begin{footnotesize}
$ P_{m,q}^{\mathcal{A}}(S^1,S^1)=\{ f\in C(S^1,S^1)|\tilde{f_{\alpha}}\in \mathcal{A}_{0,1}^{\tilde{q},\tilde{q}+m}(I,\mathbb{R}),\;p\circ\tilde{f_{\alpha}}(0)=q\}.$
\end{footnotesize}
\end{center}
\begin{thm}
For $m\in \mathbb{Z}$ and  $q\in S^1 $, $P_{m,q}^{\mathcal{A}}(S^1,S^1)$ is dense in $C_m^q(S^1,S^1)$.
\end{thm}
\begin{proof}
Proof is similar to the proof of Theorem \ref{rb}.
\end{proof}

		

\end{document}